\documentclass[oneside,english]{amsart}
\usepackage[T1]{fontenc}
\usepackage{amsthm}
\usepackage{amssymb}
\usepackage{esint}
\usepackage[all]{xy}
\usepackage{lmodern}

\makeatletter

\numberwithin{equation}{section}
\numberwithin{figure}{section}
\theoremstyle{plain}
\newtheorem{thm}{\protect\theoremname}[section]
  \theoremstyle{plain}
  \newtheorem{prop}[thm]{\protect\propositionname}
  \theoremstyle{definition}
  \newtheorem{defn}[thm]{\protect\definitionname}
  \theoremstyle{remark}
  \newtheorem{rem}[thm]{\protect\remarkname}
  \theoremstyle{plain}
  \newtheorem{lem}[thm]{\protect\lemmaname}
  \theoremstyle{definition}
  
  \newtheorem{cor}[thm]{\protect\corollaryname}
  \theoremstyle{definition}

\makeatother

\usepackage{babel}
  \providecommand{\definitionname}{Definition}
  \providecommand{\examplename}{Example}
  \providecommand{\lemmaname}{Lemma}
  \providecommand{\propositionname}{Proposition}
  \providecommand{\remarkname}{Remark}
  \providecommand{\corollaryname}{Corollary}
\providecommand{\theoremname}{Theorem}

\begin{document}

\title{The homotopy theory of bialgebras over pairs of operads}

\address{Laboratoire Paul Painlevé, Université de Lille 1, Cité Scientifique,
59655 Villeneuve d'Ascq Cedex, France}
\email{Sinan.Yalin@ed.univ-lille1.fr}

\author{Sinan Yalin}
\begin{abstract}
We endow the category of bialgebras over a pair of operads
in distribution with a cofibrantly generated model category structure.
We work in the category of chain complexes over a field of
characteristic zero. We split our construction in two steps.
In the first step, we equip coalgebras over an operad with a cofibrantly generated model category structure.
In the second step we use the adjunction between bialgebras and coalgebras via the free algebra functor.
This result allows us to do classical homotopical algebra in various
categories such as associative bialgebras, Lie bialgebras or Poisson bialgebras
in chain complexes.


\textit{AMS} : 18G55 ; 18D50.

\end{abstract}

\maketitle

\tableofcontents{}

\section*{Introduction}

The goal of this paper is to define a model category structure for the categories of bialgebras
governed by operads in distribution. The work of Drinfeld on quantum groups (see \cite{Dri1} and \cite{Dri2})
has initiated the study of bialgebra structures where the product and the coproduct belong to various
types of algebras. Besides the classical Hopf algebras, examples include their non-commutative non-cocommutative
variant, Lie bialgebras, and Poisson bialgebras. Applications range from knot theory in topology to integrable
systems in mathematical physics.
The theory of operads in distribution, introduced by Fox and Markl in \cite{FM}, provides a convenient generalization
of the classical categories of bialgebras defined by products and coproducts in distribution.
The general idea is that there is an operad encoding the operations (where we have several inputs and a single output) and another operad
encoding the cooperations (a single input and several outputs). The distributive law then formalizes the interplay
between these operads, i.e the compatibilities between operations and cooperations.
We refer the reader to \cite{Lod} for a detailed survey providing many examples of these generalized bialgebras.
One may then wonder how to transpose homotopical algebra methods in this setting, as it has been done successfully for algebras over operads.
The aim of this paper is precisely to define a closed model category structure for bialgebras over operads in distribution.

The existence of a cofibrantly generated model category structure on algebras over a suitable operad
is a classical result, see \cite{Hin}. When working over a field of characteristic zero, such a structure
exists for any operad. Let $Ch_{\mathbb{K}}^+$ be the category of positively graded chain complexes. We denote by
$^P Ch_{\mathbb{K}}^+$ the category of $P$-coalgebras in $Ch_{\mathbb{K}}^+$.
To simplify, we only consider operads in the category of $\mathbb{K}$-vector spaces $Vect_{\mathbb{K}}$ in this paper.
We use that $Ch_{\mathbb{K}}^+$ is tensored over $Vect_{\mathbb{K}}$ (in a way compatible with respect to internal tensor structures)
to give a sense to the notion of $P$-coalgebra in this category $Ch_{\mathbb{K}}^+$.
We use similar conventions when we deal with algebras over operads.
In a first step we establish the existence of a cofibrantly generated model category structure for coalgebras over an operad:
\begin{thm}
Let $P$ be an operad in $Vect_{\mathbb{K}}$ such that $P(0) = 0$, $P(1) = \mathbb{K}$,
and the vector spaces $P(n)$, $n>1$ are finite dimensional.
The category of $P$-coalgebras $^P Ch_{\mathbb{K}}^+$ inherits a cofibrantly generated model category structure
such that a morphism $f$ of $^P Ch_{\mathbb{K}}^+$ is

(i) a weak equivalence if $U(f)$ is a weak equivalence in $Ch_{\mathbb{K}}^+$;

(ii) a cofibration if $U(f)$ is a cofibration in $Ch_{\mathbb{K}}^+$;

(iii) a fibration if $f$ has the right lifting property with respect to acyclic cofibrations.
\end{thm}
Note that an analoguous result has been proven in \cite{Smi} in the context of unbounded
chain complexes. We follow another simpler approach. We do not address the same level of generality, 
but we obtain a stronger result. To be more precise, in constrast with \cite{Smi}, we obtain a cofibrantly generated structure.
These generating cofibrations are crucial to transfer the model structure on bialgebras.
Moreover, we do not need the hypothesis considered in \cite{Smi}
about the underlying operad (see \cite{Smi}, condition 4.3). Our method is close to the ideas of \cite{GG}.
Such a result also appears in \cite{AC}, but for coalgebras over a quasi-free cooperad.

Then we transfer this cofibrantly generated model structure to the category of $(P,Q)$-bialgebras.
We denote by $^Q_P Ch_{\mathbb{K}}^+$ the category of $(P,Q)$-bialgebras in $Ch_{\mathbb{K}}^+$,
where $P$ encodes the operations and $Q$ the cooperations. We use an adjunction
\[
P: {}^QCh_{\mathbb{K}}^+\rightleftarrows {}^Q_PCh_{\mathbb{K}}^+:U
\]
to perform the transfer of model structure and obtain our main theorem:
\begin{thm}
The category of $(P,Q)$-bialgebras $^Q_P Ch_{\mathbb{K}}^+$ inherits a cofibrantly generated model category structure
such that a morphism $f$ of $^Q_P Ch_{\mathbb{K}}^+$ is

(i) a weak equivalence if $U(f)$ is a weak equivalence in $^Q Ch_{\mathbb{K}}^+$
(i.e a weak equivalence in $Ch_{\mathbb{K}}^+$ by definition of the model structure on $^Q Ch_{\mathbb{K}}^+$);

(ii) a fibration if $U(f)$ is a fibration in $^Q Ch_{\mathbb{K}}^+$;

(iii) a cofibration if $f$ has the left lifting property with respect to acyclic fibrations.
\end{thm}

The strategy of the proofs is the following.
First we prove Theorem 0.1.
For this aim we construct, for any $P$-coalgebra $A$, a cooperad $U_{P^*}(A)$ called its enveloping cooperad,
whose coalgebras are the $P$-coalgebras over $A$. It expresses
the coproduct of $A$ with a cofree coalgebra in terms of the evaluation of the
associated enveloping cooperad functor. Axioms M2 and M3 are obvious. Axioms M1 is proved
in an analogue way than in the case of algebras. The main difficulty lies in the proofs of
M4 and M5. We use proofs inspired from that of \cite{GG} and adapted to our
operadic setting. The enveloping cooperad plays a key role here.
In order to produce the desired factorization axioms, our trick here
is to use a slightly modified version of the usual small object argument. We use smallness
with respect to injections systems running over a certain ordinal.

Then we prove our main result, Theorem 0.2. Axioms M1, M2 and M3 are easily checked, and axiom M4 follows
from arguments similar to the case of coalgebras. The main difficulty is the proof of M5. We use mainly the small object
argument for smallness with respect to injections systems, combined with a result about cofibrations in algebras over an operad due to Hinich \cite{Hin}.

Let us point out that in both cases, we cannot use the usual simplifying hypothesis of smallness with respect to all morphisms.
This is due to the fact that Lemma 2.16, giving an essential smallness property for coalgebras, is only true for
smallness with respect to injection systems.

\textit{Organization}: the overall setting is reviewed in section
1. We suppose some prerecquisites concerning operads (see \cite{LV})
and give definitions of algebras and coalgebras over an operad.
Then we define distributive laws from the monadic viewpoint, following \cite{FM}.
Examples of monads and comonads include operads and cooperads.
We recall some basic facts about the small object argument in cofibrantly generated
model categories, in order to fix useful notations for the following.

The heart of this paper consists of sections 2 and 3, devoted to the proofs of Theorem
0.1 and 0.2. The proof of Theorem 0.1 heavily relies on the notion of enveloping cooperad,
which is defined in 2.2. In 2.3, we follows the argument line of \cite{GG}, checking carefully
where modifications are needed to work at our level of generality.
Theorem 0.2 is proved in section 3, by using adjunction properties to transfer the model structure
obtained in Theorem 0.1. The crux here is a small object argument with respect to
systems of injections of coalgebras.

\section{Recollections}

In this section, we first list some notions and facts about operads and algebras over operads.
Then we review the interplay between monads and comonads by means of distributive laws and make the link with operads.
It leads us to the crucial definition of bialgebras over pairs of operads in distribution.
Finally, we recall a classical tool of homotopical algebra, namely the small object argument, aimed to produce factorizations
in model categories. The material of this section is taken from \cite{LV}, \cite{FM} and \cite{Hov}.

\subsection{Algebras and coalgebras over operads}

For the moment, we work in the category of non-negatively graded chain complexes $Ch_{\mathbb{K}}$,
where $\mathbb{K}$ is a field of characteristic zero (but we still assume that our operads belong
to $Vect_{\mathbb{K}}$).
We adopt most conventions of \cite{LV} and freely use the notations of this reference.
The only exception is the name $\Sigma$-module which denotes the notion
of $\mathbb{S}$-module of \cite{LV}. This convention is more usual in topology.
We also refer the reader to \cite{LV} for the definitions of operads and cooperads.

Operads are used to parametrize various kind of algebraic structures.
Fundamental examples of operads include the operad $As$ encoding associative algebras,
the operad $Com$ of commutative algebras, the operad $Lie$ of Lie algebras and the operad
$Pois$ of Poisson algebras.
There exists several equivalent approaches for the
definition of an algebra over an operad. We will use the following one which we recall for convenience:
\begin{defn}\textit{(the monadic approach)}
Let $(P,\gamma,\iota)$ be an operad, where $\gamma$ is the composition product and $\iota$ the unit.
A $P$-algebra is a complex $A$ endowed
with a linear map $\gamma_{A}:P(A)\rightarrow A$ such that
the following diagrams commute 
\[
\xymatrix{(P\circ P)(A)\ar[r]^{P(\gamma_{A})}\ar[d]_{\gamma(A)} & P(A)\ar[d]^{\gamma_{A}}\\
P(A)\ar[r]_{\gamma_{A}} & A
}
\]

\[
\xymatrix{A\ar[r]^{\iota(A)}\ar[rd]_{=} & P(A)\ar[d]^{\gamma_{A}}\\
 & A
}
.
\]
\end{defn}
We will denote $_P Vect_{\mathbb{K}}$ the category of $P$-algebras in vector spaces
and $_P Ch_{\mathbb{K}}$ the category of $P$-algebras in non-negatively graded chain complexes.

For every complex $V$, we can equip $P(V)$ with a $P$-algebra structure
by setting $\gamma_{P(V)}=\gamma (V):P(P(V)) \rightarrow P(V)$. As a consequence of the definition,
we thus get the free $P$-algebra functor:
\begin{prop}
(see \cite{LV}, Proposition 5.2.6) The $P$-algebra $(P(V),\gamma(V))$ equiped with the map
$\iota(V):I(V)=V \rightarrow P(V)$ is the free $P$-algebra on $V$.
\end{prop}

There is also a notion of coalgebra over a cooperad:
\begin{defn}\textit{(the comonadic approach)}
Let $(C,\Delta,\epsilon)$ be a cooperad, where $\Delta$ is the decomposition product and $\epsilon$ the counit
(they define on $C$ a structure of comonoid). A $C$-coalgebra is a complex $X$ equiped with a linear application
$\rho_X:X\rightarrow C(X)$ such that the following diagrams commute:
\[
\xymatrix{X\ar[r]^{\rho_X}\ar[d]_{\rho_X} & C(X)\ar[d]^{\Delta\rho_X}\\
C(X)\ar[r]_{C(\rho_X)} & C(C(X))
}
\]

\[
\xymatrix{X\ar[r]^{\rho}\ar[dr]_{=} & C(X)\ar[d]^{\epsilon(X)}\\
 & X
}.
\]
\end{defn}

We can go from operads to cooperads and vice-versa by dualization.
Indeed, if $C$ is a cooperad, then the $\Sigma$-module $P$ defined by $P(n)=C(n)^*=Hom_{\mathbb{K}}(C(n),\mathbb{K})$
form an operad. Conversely, suppose that $\mathbb{K}$ is of characteristic zero and $P$ is an operad such that 
each $P(n)$ is finite dimensional. Then the $P(n)^*$ form a cooperad, in the sense of \cite{GJ} and \cite{LV}.
The additional hypotheses are needed because we have to use, for finite dimensional vector spaces $V$ and $W$, the isomorphism
$(V\otimes W)^*\cong V^*\otimes W^*$ to define properly the cooperad coproduct.

We also give the definition of coalgebras over an operad:
\begin{defn}
(1) Let $P$ be an operad. A $P$-coalgebra is a complex $C$ equiped with linear applications
$\rho_n:P(n)\otimes C \rightarrow C^{\otimes n}$ for every $n\geq0$. These maps are $\Sigma_n$-equivariant
and associative with respect to the operadic compositions.

(2) Each $p\in P(n)$ gives rise to a cooperation $p^*:C\rightarrow C^{\otimes n}$.
The coalgebra $C$ is usually said to be conilpotent if for each $c\in C$, there exists $N\in\mathbb{N}$
so that $p^*(c)=0$ when we have $p\in P(n)$ with $n>N$.
\end{defn}
If $\mathbb{K}$ is a field of characteristic zero and the $P(n)$ are finite dimensional, then
it is equivalent to define a $P$-coalgebra via a family of applications
$\overline{\rho}_n:C\rightarrow P(n)^*\otimes_{\Sigma_n} C^{\otimes n}$.

Now suppose that $P$ is an operad such that the $P(n)$ are finite dimensional, $P(0)=0$ and $P(1)=\mathbb{K}$.
Then we have a cofree conilpotent $P$-coalgebra functor, which is by definition
the right adjoint to the forgetful functor and is given by the comonad associated to the cooperad $P^*$:
\begin{thm}
Let $V$ be an object of $Ch_{\mathbb{K}}$. Then
\[
P^*(V)=\bigoplus_{r=1}^{\infty} P(r)^* \otimes_{\Sigma_r} V^{\otimes r}
\]
inherits a $P$-coalgebra structure and forms the cofree conilpotent $P$-coalgebra.
\end{thm}

The conilpotence condition is automatically fulfilled when we deal with operads in $Vect_{\mathbb{K}}$
and coalgebras in $Ch_{\mathbb{K}}^+$,
because these hypotheses ensure that the morphisms $P(n)\otimes C_d\rightarrow (C^{\otimes n})_d$
are zero for $n>d$.
In the next sections we will use these assumptions and just say $P$-coalgebra
to refer to a conilpotent $P$-coalgebra.
Under these assumptions, we also get as a corollary of this theorem an equivalence between the notion of coalgebra
over the operad $P$ (Definition 1.4) and the notion of coalgebra over the cooperad $P^*$ (Definition 1.3).

\subsection{Monads, comonads and distributive laws}

In certain cases, bialgebras can be parametrized by a pair of operads in the following
way: one operad encodes the operations, the other encodes the cooperations, such that
the concerned bialgebra forms an algebra over the first operad and a coalgebra over the second operad.
The compatibility relations between operations and cooperations are formalized by
the notion of mixed distributive law between the two operads. The Schur functor associated to an operad
forms a monad, and the Schur functor associated to a cooperad forms a comonad. The mixed distributive law induces
a distributive law between this monad and this comonad. We briefly review the notion
of distributive law in the monadic setting. We refer the reader to \cite{FM} for definitions of monads, comonads, their algebras
and coalgebras.

Let $\mathcal{C}$ be a category. Suppose we have in $\mathcal{C}$ a monad $(P,\gamma,\iota)$
and a comonad $(Q^*,\delta,\epsilon)$. We would like to make $P$ and $Q^*$
compatible, that is to define $Q^*$-coalgebras in $P$-algebras or
conversely $P$-algebras in $Q^*$-coalgebras. This compatibility is formalized
by the notion of mixed distributive law \cite{FM}:

\begin{defn}
A mixed distributive law $\lambda:P Q^*\rightarrow Q^* P$
between $P$ and $Q^*$ is a natural transformation satisfying the following conditions:

(i)$\Lambda\circ\gamma Q^*=Q^*\gamma\circ\Lambda$

(ii)$\delta P\circ\Lambda=\Lambda\circ P\delta$

(iii)$\lambda\circ\iota Q^*=Q^*\iota$

(iv)$\epsilon P\circ\lambda=P\epsilon$

where the $\Lambda:P^m (Q^*)^n\rightarrow (Q^*)^n P^m$, for every
natural integers $m$ and $n$, are the natural transformations obtained by iterating $\lambda$.
For instance, for $m=2$ and $n=3$ we have
\[
\xymatrix{P^2 (Q^*)^3\ar[r]^{P\lambda (Q^*)^2}
& P Q^* P (Q^*)^2\ar[r]^{\lambda^2 Q^*}
& Q^* P Q^* P Q^*\ar[r]^{Q^*\lambda^2}
& (Q^*)^2 P Q^* P\ar[r]^{(Q^*)^2\lambda P}
& (Q^*)^3 P^2}
\]
\end{defn}
These conditions allow us to lift $P$ as an endofunctor of the category $Q^*-Coalg$
of $Q^*$-coalgebras and $Q^*$ as an endofunctor of the category $P-Alg$ of $P$-algebras.
These notations are chosen to emphasize the fact that later, the monad $P$ will
correspond to an operad $P$ and the comonad $Q^*$ to an operad $Q$ (which gives a comonad
$Q^*$ by dualization and the finiteness hypothesis).

Then we can define the notion of bialgebra over a pair (monad,comonad) endowed with
a mixed distributive law:
\begin{defn}
(a) Given a monad $P$, a comonad $Q^*$ and a mixed distributive law
$\lambda:PQ^*\rightarrow Q^*P$, a $(P,Q^*)$-bialgebra
$(B,\beta,b)$ is an object $B$ of $\mathcal{C}$ equiped with two morphisms $\beta:P(B)\rightarrow B$
and $b:B\rightarrow Q^*(B)$ defining respectively a $P$-algebra structure and a
$Q^*$-coalgebra structure. Furthermore, the maps $\beta$ and $b$ satisfy a compatibility
condition expressed through the commutativity of the following diagram:
\[
\xymatrix{ P(Q^*(B)) \ar[d]_{\lambda(B)} & P(B)\ar[l]_{P(b)} \ar[dd]^{\beta} \\
Q^*(P(B)) \ar[d]_{Q^*(\beta)} & \\
Q^*(B) & B \ar[l]^b }
\]
(b) A morphism of $(P,Q^*)$-bialgebras is a morphism of $\mathcal{C}$ which is both a morphism
of $P$-algebras and a morphism of $Q^*$-coalgebras.

The category of $(P,Q^*)$-bialgebras is denoted $(P,Q^*)-Bialg$.
\end{defn}

\begin{rem}
The application $Q^* (\beta)\circ\lambda(B)$ endows $Q^*(B)$ with a $P$-algebra structure,
and the application $\lambda(B)\circ P(b)$ endows $P(B)$ with a $Q^*$-coalgebra structure.
Moreover, given these two structures, the compatibility diagram of Definition 1.7 shows that $\beta$ is a morphism
of $Q^*$-coalgebras and $b$ a morphism $P$-algebras. The $(P,Q^*)$-bialgebras
can therefore be considered as $Q^*$-coalgebras in the category $P-Alg$ of $P$-algebras
or as $P$-algebras in the category $Q^*-Coalg$ of $Q^*$-coalgebras.
\end{rem}

In the case of operads, there is a notion of mixed distributive law between two operads,
defined by explicite formulae for which we refer the reader to \cite{FM}.
The link with the monadic distributive law is given by the following theorem:

\begin{thm}(cf. \cite{FM})
Let $P$ and $Q$ be two operads endowed with a mixed distributive law.
Then the monad $P(-)$ and the comonad $Q^*(-)$ inherits a distributive law in the sense of Definition 1.6
induced by this mixed distributive law.
\end{thm}
We define a $(P,Q)$-bialgebra as a bialgebra over these monads in distribution in the sense of Definition 1.7.
Suppose that $P$ and $Q$ are operads in $Vect_{\mathbb{K}}$ such that the $Q(n)$ are finite dimensional,
$Q(0)=0$ and $Q(1)=\mathbb{K}$.
Then we know that the notions of $Q$-coalgebras and $Q^*$-coalgebras coincide.
A $(P,Q)$-bialgebra is thus equiped with a $P$-algebra structure, a $Q$-coalgebra structure
and compatibilities with respect to the distributive law.
The operadic distributive law formalizes the interplay between algebraic operations
and coalgebraic cooperations of the bialgebra.

Let us finally note that if $B$ is a $(P,Q)$-coalgebra, then, as a corollary of Theorem 1.9 and Remark 1.8,
the free $P$-algebra $P(B)$ has a natural structure of $Q$-coalgebra and
the cofree $Q$-coalgebra $Q^*(B)$ has a natural structure of $P$-algebra.

\subsection{Model categories and the small object argument}

Model categories are the natural setting to do homotopical algebra.
This means that they encode well defined notions of cylinder objects and path objects,
homotopy classes, non-abelian cohomology theories and non abelian functor derivation (Quillen's derived functors).
We will just recall here some facts about cofibrantly generated model categories and the small object argument,
for the purpose to fix conventions and the definition of objects used in our constructions.
We refer the reader to the classical reference \cite{Qui}, but also to \cite{DS} for a well-written introduction
to model categories and their homotopy theories, as well as \cite{Hov} and \cite{Hir} to push the analysis further.
Let us briefly recall the small object argument,
which is a general and useful way to produce factorizations with lifting properties with respect
to a given class of morphisms.
We just sum up the construction given in \cite{Hov} without detailing the process.
We refer the reader for section 2.1.1 of \cite{Hov} for recollections about ordinals, cardinals
and transfinite composition.

Suppose that $\mathcal{C}$ is a category admitting small colimits. Let $\lambda$ be an ordinal.
A $\lambda$-sequence is a colimit preserving functor $B:\lambda\rightarrow\mathcal{C}$, written as
\[
B(0)\rightarrow B(1)\rightarrow...\rightarrow B(\beta)\rightarrow...
\]
Now let us fix an object $A$ of $\mathcal{C}$, a collection $\mathcal{D}$ of morphisms of $\mathcal{C}$
and a cardinal $\kappa$.
\begin{defn}
(1) The object $A$ is $\kappa$-small with respect to $\mathcal{D}$ if for
all $\kappa$-filtered ordinals $\lambda$ and all $\lambda$-sequences
\[
B(0)\rightarrow B(1)\rightarrow...\rightarrow B(\beta)\rightarrow...
\]
such that each map $B(\beta)\rightarrow B(\beta+1)$ is in $\mathcal{D}$ for $\beta+1<\lambda$,
the canonical induced map 
\[
colim_{\beta<\lambda}Mor_{\mathcal{C}}(A,B(\beta))\rightarrow Mor_{\mathcal{C}}(A,colim_{\beta<\lambda} B(\beta))
\]
is a bijection.

(2) The object $A$ is small if it is $\kappa$-small with respect to all morphisms of $\mathcal{C}$
for a certain cardinal $\kappa$.
\end{defn}
A fundamental example of small object is the following:
\begin{lem}(\cite{Hov}, Lemma 2.3.2)
Every chain complex over a ring is small.
\end{lem}

Let $\kappa$ be a cardinal.
Let $\mathcal{F}=\{f_i:A_i\rightarrow B_i\}_{i\in I}$ be a set of morphisms of $\mathcal{C}$. We consider a morphism
$g:X\rightarrow Y$ of $\mathcal{C}$ for which we want to produce a factorization $X\rightarrow X'\rightarrow Y$, such
that $X'\rightarrow Y$ has the right lifting property with respect to the morphisms of $\mathcal{F}$.
There is a recursive construction providing the following commutative diagram:
\[
\xymatrix{ X \ar[d]_g \ar[r]^-{i_1} & G^1(\mathcal{F},g) \ar[d]_{g_1} \ar[r]^-{i_2} &...\ar[r]^-{i_{\beta}} & G^{\beta}(\mathcal{F},g)\ar[d]_{g_{\beta}}
\ar[r]^-{i_{\beta+1}} &... \\ Y \ar@{=}[r] & Y \ar@{=}[r] &...\ar@{=}[r] & Y \ar@{=}[r] &...}
\]
where the upper line is a $\lambda$-sequence for a certain $\kappa$-filtered ordinal $\lambda$.
In this recursive procedure, each $i_{\beta}$ is obtained by a pushout of the form
\[
\xymatrix{\bigoplus_{\alpha}A_{\alpha} \ar[d]_{\bigoplus_{\alpha}f_{\alpha}} \ar[r] & G^{\beta-1}(\mathcal{F},g)
\ar[d]^{i_{\beta}} \\ \bigoplus_{\alpha}B_{\alpha} \ar[r] & G^{\beta}(\mathcal{F},g) }
\]
where the $f_{\alpha}$ are morphisms of $\mathcal{F}$. The category $\mathcal{C}$ is supposed to admit
small colimits, so we can consider the infinite composite $i_{\infty}:X\rightarrow G^{\infty}(\mathcal{F},g)$ of the sequence of maps
\[
\xymatrix{ X \ar[r]^-{i_1} & G^1(\mathcal{F},p) \ar[r]^-{i_2} &...\ar[r]^-{i_{\beta}} & G^{\beta}(\mathcal{F},p)
\ar[r]^-{i_{\beta+1}} &...\ar[r] & G^{\infty}(\mathcal{F},p)}
\]
where $G^{\infty}(\mathcal{F},p)=colim_{\beta<\lambda}G^{\beta}(\mathcal{F},p)$ is the colimit of this system.
A morphism like $i_{\infty}:X\rightarrow G^{\infty}(\mathcal{F},g)$, that is obtained by a (possibly transfinite) composition of pushouts
of maps of $\mathcal{F}$, is called a relative $\mathcal{F}$-cell complex.
By universal property of the colimit, the morphism $g$ has a factorization $g=g_{\infty}\circ i_{\infty}$ where
$g_{\infty}:G^{\infty}(\mathcal{F},g)\rightarrow Y$.

\begin{thm}
(cf. \cite{Hov}, Theorem 2.1.14)
In the preceding situation, suppose that for every $i\in I$, the object $A_i$ is $\kappa$-small in $\mathcal{C}$
with respect to relative $\mathcal{F}$-cell complexes.
Then the morphism $g_{\infty}$ defined above has the right lifting property with respect to the morphisms of $\mathcal{F}$.
\end{thm}

In the remaining sections of our paper, in order to deal with morphisms of the form of $i_{\infty}$
we will need the following properties:

\begin{lem}(see \cite{Hir})
(1) Let us consider a pushout of the form
\[
\xymatrix{ K \ar[r]^f \ar[d]_i & K' \ar[d]^j \\ L \ar[r]_g & L' }
\]
in a category $\mathcal{C}$ admitting small colimits. Suppose that $i$ has the left lifting property
with respect to a given family $\mathcal{F}$ of morphisms of $\mathcal{C}$. Then $j$ has also the
left lifting property with respect to $\mathcal{F}$. Another way to state this result is to say that
the left lifting property with respect to a given family of morphisms is invariant under cobase change.

(2) Let us consider a $\lambda$-sequence
\[
\xymatrix{ G^0 \ar[r]^-{i_1} & G^1 \ar[r]^-{i_2} &...\ar[r]^-{i_{\beta}} & G^{\beta} \ar[r]^-{i_{\beta+1}} &...\ar[r] & colim_{\beta<\lambda} G^{\beta}=G^{\infty}}.
\]
Let us note $i_{\infty}:G^0\rightarrow G^{\infty}$ the transfinite composite of the $i_{\beta}$. If for every ordinal $\beta<\lambda$, the morphism
$i_{\beta}$ has the left lifting property with respect to a given family $\mathcal{F}$ of morphisms of $\mathcal{C}$, then so does
$i_{\infty}$.
\end{lem}

It is time now to give a concrete example of model category. Of course, topological spaces provide the initial
example from which the theory of model categories arised. However, the example we will use to illustrate these notions
is that of chain complexes. This choice is motivated by two reasons. Firstly, this will be the base category for
the remaining part of our paper. Secondly, the model category structures of algebras and coalgebras over operads
will be transfered from this one via adjunctions.

\begin{thm}
(cf. \cite{DS}, Theorem 7.2)
The category $Ch_{\mathbb{K}}$ of chain complexes over a field $\mathbb{K}$ forms a cofibrantly generated model category
such that a morphism $f$ of $Ch_{\mathbb{K}}$ is

(i) a weak equivalence if for every $n\geq 0$, the induced map $H_n(f)$ in homology is an isomorphism.

(ii) a fibration if for every $n> 0$, the map $f_n$ is surjective.

(iii) a cofibration if for every $n\geq 0$, the map $f_n$ is injective.
\end{thm}

To conclude this section, let us note that the commonly used small object argument, for instance to prove theorem 1.14
but also for various other examples (like algebras over operads), is Proposition 7.17 in \cite{DS}. That is,
use the simplifying assumption of smallness with respect to all morphisms.
We will need a more refined version for coalgebras an operad.

\medskip{}

\section{The model category of coalgebras over an operad}

We work in the full subcategory $Ch_{\mathbb{K}}^+$ of $Ch_{\mathbb{K}}$
whose objects are the chain complexes $C$ such that $C_0=0$, i.e the positively graded chain complexes. The category $Ch_{\mathbb{K}}^+$
is actually a model subcategory of $Ch_{\mathbb{K}}$. We suppose that $P$ is an operad
in $Vect_{\mathbb{K}}$ such that the $P(n)$ are finite dimensional, $P(0)=0$ and $P(1)=\mathbb{K}$. Note that the commonly used operads
satisty this hypothesis, for instance $As$ (for the associative algebras), $Com$ (for the commutative associative algebras), $Lie$ 
(for the Lie algebras), $Pois$ (for the Poisson algebras). There are two difficulties appearing here. Firstly, our
operad is not defined exactly in the same category as our algebras. Secondly, the category $Ch_{\mathbb{K}}^+$ inherits the symmetric
monoidal structure of $Ch_{\mathbb{K}}$ but not the unit (which is $\mathbb{K}$ concentrated in degree $0$).
However, $Vect_{\mathbb{K}}$ acts on $Ch_{\mathbb{K}}^+$ via the usual tensor product of chain complexes,
when we identify $Vect_{\mathbb{K}}$ with the subcategory of $Ch_{\mathbb{K}}$ consisting in complexes
concentrated in degree $0$ .
The convenient notion to deal with such situations is the one of symmetric monoidal category over a base category.
We refer the reader to \cite{Fre3}, chapter 1, for a precise definition and the associated properties.
In our case, we work in the reduced symmetric monoidal category $Ch_{\mathbb{K}}^+$ over the base $Vect_{\mathbb{K}}$
(see also \cite{Fre3}, 1.1.17). As shown in \cite{Fre3}, all the usual definitions and properties of operads and
their algebras hold in the reduced setting. The situation is analogous for cooperads and their coalgebras.
The model category structure on coalgebras is given by the following theorem:
\begin{thm}
The category of $P$-coalgebras $^P Ch_{\mathbb{K}}^+$ inherits a cofibrantly generated model category structure
such that a morphism $f$ of $^P Ch_{\mathbb{K}}^+$ is

(i) a weak equivalence if $U(f)$ is a weak equivalence in $Ch_{\mathbb{K}}^+$;

(ii) a cofibration if $U(f)$ is a cofibration in $Ch_{\mathbb{K}}^+$;

(iii) a fibration if $f$ has the right lifting property with respect to acyclic cofibrations.
\end{thm}

From now on, we use the numbering of \cite{Hir} (Definition 7.1.3) for the axioms of a model category.
The three class of morphisms defined in this theorem are clearly stable by composition and contain the identity maps.
Axioms M2 and M3 are clear, and M4 (ii) is obvious by definition of the fibrations. It remains to prove axioms
M1, M4 (i) and M5.

\subsection{Proof of M1}

The forgetful functor creates the small colimits. The proof of this fact is exactly the same as the proof
of the existence of small limits in the $P$-algebras case. To prove the existence of small limits in $^P Ch_{\mathbb{K}}^+$,
we use the following categorical result:
\begin{thm}(see \cite{MLa2})
Let $\mathcal{C}$ be a category. If $\mathcal{C}$ admits the coreflexive equalizers of every pair of arrows
and all small coproducts, then $\mathcal{C}$ admits all the small limits.
\end{thm}

Now let us prove that $^P Ch_{\mathbb{K}}^+$ admits the coreflexive equalizers and the small products.
\begin{lem}
Let $d^0,d^1:A\rightarrow B$ be two morphisms in $^P Ch_{\mathbb{K}}^+$ and $s^0:B\rightarrow A$
a morphism of $Ch_{\mathbb{K}}^+$  such that $s^0\circ d^0 = s^0\circ d^1 = id_A$.
Then $ker(d^0-d^1)$ is the coreflexive equalizer of $(d^0,d^1)$ in $^P Ch_{\mathbb{K}}^+$.
\end{lem}

\begin{proof}
The subspace $ker(d^0-d^1)\subset A$ is the coreflexive equalizer of $(d^0,d^1)$ in $Ch_{\mathbb{K}}^+$.
Moreover, it is a sub-$P$-coalgebra of $A$ and the inclusion is obviously a $P$-coalgebras morphism.
Indeed, let $\alpha\in A$ such that $d^0(\alpha)=d^1(\alpha)$, i.e $\alpha\in ker(d^0-d^1)$.
We want to prove that $ker(d^0-d^1)$ is stable under the cooperations of $A$. That is, for every cooperation
$p^*:A\rightarrow A^{\otimes n}$ associated to $p\in P(n)$, the image $p^*(\alpha)$ actually lies in $ker(d^0-d^1)^{\otimes n}$.
We have
\[
ker(d^0-d^1)^{\otimes n} = \bigcap_i A\otimes ...\otimes ker(d^0-d^1)\otimes ...\otimes A.
\]
Let $p^*:A\rightarrow A^{\otimes n}$ be a cooperation associated to $p\in P(n)$.
Then the following equalities hold:
\begin{eqnarray*}
id\otimes...\otimes d^0\otimes...\otimes id \circ p^*(\alpha) & = & (s^0\circ d^0)\otimes...\otimes d^0\otimes...\otimes (s^0\circ d^0) \circ
p^*(\alpha) \\
 & = & s^0\otimes...\otimes id\otimes...\otimes s^0\circ (d^0)^{\otimes n}\circ p^*(\alpha) \\
 & = & s^0\otimes...\otimes id\otimes...\otimes s^0\circ p^*\circ d^0(\alpha) \\
 & = & s^0\otimes...\otimes id\otimes...\otimes s^0\circ p^*\circ d^1(\alpha) \\
 & = & id\otimes...\otimes d^1\otimes...\otimes id \circ p^*(\alpha).
\end{eqnarray*}
The first line holds because of the equality $s^0\circ d^0=id$ satisfied by hypothesis.
The third line comes from the fact that $d^0$ is a $P$-coalgebras morphism.
The fourth line follows from our assumption that $\alpha\in ker(d^0-d^1)$.
The last line is obtained by following the preceding arguments in the converse direction.
According to our decomposition of $ker(d^0-d^1)^{\otimes n}$, it precisely means that $p^*(\alpha)\in ker(d^0-d^1)^{\otimes n}$,
which concludes the proof.
\end{proof}

\begin{lem}
Let $\{R_i\}_{i\in I}$ be a set of $P$-coalgebras. Let us set
\[
d^0=P^*(\bigoplus \rho_{R_i}):P^*(\bigoplus R_i)\rightarrow P^*(\bigoplus P^*(R_i))
\]
and
\[
d^1=\pi\circ \Delta(\bigoplus R_i):P^*(\bigoplus R_i)\rightarrow P^*(\bigoplus P^*(R_i))
\]
where $\pi:P^*(P^*(\bigoplus R_i))\rightarrow P^*(\bigoplus P^*(R_i))$ is the canonical projection and $\Delta$
the comultiplication of the comonad $(P^*,\Delta,\epsilon)$.
Then $\times R_i=ker(d^0-d^1)$ is the product of the $R_i$ in $^P Ch_{\mathbb{K}}^+$.
\end{lem}

\begin{proof}
We prove the lemma in the case of two $P$-coalgebras $R$ and $S$. The proof is the same in the general case.
Let us set
\[
s^0=P^*(\epsilon(R)\oplus\epsilon(S)):P^*(P^*(R)\oplus P^*(S))\rightarrow P^*(R\oplus S),
\]
then $s^0\circ d^0=s^0\circ d^1=id$. According to Lemma 2.3, the space $ker(d^0-d^1)$ is the coreflexive
equalizer of $(d^0,d^1)$ in $^P Ch_{\mathbb{K}}^+$. Let $X$ be a $P$-coalgebra. Two linear maps
$u:X\rightarrow R$ and $v:X\rightarrow S$ induce a map $(u,v):X\rightarrow R\oplus S$.
This map admits a unique factorization through $P^*(R\oplus S)$ to give a $P$-coalgebras morphism
$\varphi_{(u,v)}:X\rightarrow P^*(R\oplus S)$ by the universal property of the cofree $P^*$-coalgebra.
This morphism admits a unique factorization through $ker(d^0-d^1)$ if and only if $u$ and $v$
are morphisms of $P$-coalgebras. By unicity of the limit this concludes our proof, since
$ker(d^0-d^1)$ satisfies the same universal property than $R\times S$.
\end{proof}

Now we introduce a crucial construction, the enveloping cooperad, which provides a handable expression
of the coproduct of a $P$-coalgebra with a cofree $P$-coalgebra. This cooperad plays a key role in the proof
of axiom M5 (i).

\subsection{Enveloping cooperad}

Let $A$ be a $P$-coalgebra. We want to construct a particular cooperad associated to $A$ and called the enveloping
cooperad of $A$. This is a "dual version" of the enveloping operad of \cite{Fre2}.
We need the following classical result :

\begin{prop}(see \cite{LV}, and Proposition 1.2.5 of \cite{Fre4})
When $\mathbb{K}$ is an infinite field, we have a fully faithful
embedding of the category of $\Sigma$-modules in the category of
endofunctors of $Vect_{\mathbb{K}}$, which associates to any $\Sigma$-module its Schur functor.
\end{prop}

We consider the $\Sigma$-module $P^*[A]$ defined by
\[
P^*[A](n)=\bigoplus_{r=1}^{\infty}P(n+r)^*\otimes_{\Sigma_r} A^{\otimes r}.
\]
We need the following lemma:
\begin{lem}
Let $A$ be a $P$-coalgebra. For every chain complex $C$ of $Ch_{\mathbb{K}}^+$ we have $P^*[A](C)\cong P^*(A\oplus C)$.
\end{lem}
\begin{proof}
We use the equality
\[
P(n)^*\otimes_{\Sigma_n} (A\oplus C)^{\otimes n} = \bigoplus_{p+q=n} P(p+q)^*\otimes_{\Sigma_p \times \Sigma_q} (A^{\otimes p} \otimes C^{\otimes q})
\]
to get
\begin{eqnarray*}
P^*(A\oplus C) & = & \bigoplus_{n\geq 0} P(n)^*\otimes_{\Sigma} (A\oplus C)^{\otimes n} \\
 & = & \bigoplus_n \bigoplus_{p+q=n} P(p+q)^*\otimes_{\Sigma_p \times \Sigma_q} (A^{\otimes p} \otimes C^{\otimes q}) \\
 & = & \bigoplus_n \bigoplus_{p+q=n} (P(p+q)^*\otimes_{\Sigma_p} A^{\otimes p}) \otimes_{\Sigma_q} C^{\otimes q}) \\
 & = & \bigoplus_q P^*[A](q)\otimes_{\Sigma_q} C^{\otimes q} \\
 & = & P^*[A](C).
\end{eqnarray*}
\end{proof}
This lemma allows us to equip $P^*[A]$ with a cooperad structure induced by the one of $P^*$.
The cooperad coproduct $P^*[A]\rightarrow P^*[A]\circ P^*[A]$ comes from a comonad coproduct
$P^*[A](-)\rightarrow P^*[A](-)\circ P^*[A](-)$ on the Schur functor $P^*[A](-)$ associated to $P^*[A]$.
This comonad coproduct is a natural map defined for every chain complex $C$ by
\begin{eqnarray*}
P^*[A](C)\cong P^*(A\oplus C) & \stackrel{\Delta (A\oplus C)}{\rightarrow} & P^*(P^*(A\oplus C)) \\
 & \stackrel{P^*(proj\circ\epsilon,id)}{\rightarrow} & P^*(A\oplus P^*(A\oplus C))\cong (P^*[A]\circ P^*[A])(C)
\end{eqnarray*}
where $\Delta$ is the coproduct of $P^*$, $\epsilon$ its counit and $proj$ the projection on the first component.

The $P$-coalgebra structure morphism $\rho_A:A\rightarrow P^*(A)$ of $A$ induces a $\Sigma$-modules morphism
\[
d^0:P^*[A]\rightarrow P^*[P^*(A)],
\]
given by
\[
d^0(n)=\bigoplus_{r=1}^{\infty}id\otimes \rho_A^{\otimes r}:\bigoplus_{r=1}^{\infty}P(n+r)^*\otimes_{\Sigma_r} A^{\otimes r}
\rightarrow \bigoplus_{r=1}^{\infty}P(n+r)^*\otimes_{\Sigma_r} P^*(A)^{\otimes r}.
\]
It comes from a natural map $d^0(-)$ defined for every chain complex $C$ by
\begin{eqnarray*}
d^0(C): P^*[A](C)\cong P^*(A\oplus C) & \stackrel{P^*(\rho_A\oplus id)}{\rightarrow} & P^*(P^*(A)\oplus C)\cong P^*[P^*[A]](C)
\end{eqnarray*}
where $\rho_A:A\rightarrow P^*(A)$ is the map defining the $P^*$-coalgebra structure of $A$.
The natural map $d^0(-)$ is by construction a morphism of comonads, so $d^0$ forms a morphism of cooperads.

The coproduct $\Delta:P^*\rightarrow P^*\circ P^*$ associated to the comonad $(P^*,\Delta,\epsilon)$
induces another morphism of $\Sigma$-modules
\[
d^1:P^*[A]\rightarrow P^*[P^*(A)]
\]
defined in the following way: we have a natural map $d^1(-)$ defined for every chain complex $C$ by
\begin{eqnarray*}
d^1(C): P^*[A](C)\cong P^*(A\oplus C) & \stackrel{\Delta (A\oplus C)}{\rightarrow} & P^*(P^*(A\oplus C)) \\
 & \stackrel{P^*(P^*(pr_A),\pi\circ P^*(pr_C))}{\rightarrow} & P^*(P^*(A)\oplus C) \\
 & \cong & P^*[P^*(A)](C)
\end{eqnarray*}
where $P^*[A](-)$ is the Schur functor associated to $P^*[A]$ and the map $\pi$ is the projection on the component of arity $1$.
According to Proposition 2.5, any natural transformation between two Schur functors determines a unique morphism between their
corresponding $\Sigma$-modules.
Thus there is a unique morphism of $\Sigma$-modules $d^1:P^*[A]\rightarrow P^*[P^*(A)]$ associated to $d^1(-)$.
The natural map $d^1(-)$ is by construction a morphism of comonads, so $d^1$ forms a morphism of cooperads.

The counit $\epsilon:P^*\rightarrow Id$ induces a morphism of $\Sigma$-modules
\[
s^0:P^*[P^*(A)]\rightarrow P^*[A]
\]
defined in the following way: we have a natural map $s^0(-)$ defined for every chain complex $C$ by
\begin{eqnarray*}
s^0(C): P^*[P^*(A)](C)\cong P^*(P^*(A)\oplus C) & \stackrel{P^*(P^*(i_A),i_C)}{\rightarrow} & P^*(P^*(A\oplus C)) \\
 & \stackrel{P^*(\epsilon (A\oplus C))}{\rightarrow} & P^*(A\oplus C)\cong P^*[A](C)
\end{eqnarray*}
where $i_A:A\rightarrow A\oplus C$ and $i_C:C\rightarrow A\oplus C$, hence the unique associated morphism of $\Sigma$-modules
$s^0:P^*[P^*(A)]\rightarrow P^*[A]$. Note that a priori $s^0$ is not a morphism of cooperads.
We finally obtain a coreflexive pair $(d^0,d^1)$ of morphisms of $\Sigma$-modules
induced by the associated reflexive pair of morphisms of Schur functors. The enveloping cooperad of $A$ is the
coreflexive equalizer
\[
\xymatrix{ U_{P^*}(A)=ker(d^0-d^1)\ar[r] & P^*[A] \ar@<1ex>[r]^{d^0} \ar@<-1ex>[r]_{d^1} & P^*[P^*(A)] \ar@/_1.3pc/[l]_{s^0} }
\]
in $\Sigma$-modules, where $d^0$ and $d^1$ are cooperad morphisms.
The fact that such a coreflexive equalizer inherits a cooperad structure from the one of $P^*[A]$
follows from arguments similar to those of the proof of Lemma 2.3.
The cooperad $U_{P^*}(A)$ is coaugmented over $P$ (by dualizing the proof of Lemma 1.2.4 in \cite{Fre2}).
One can even go further and prove that the category of $U_{P^*}(A)$-coalgebras is equivalent to the category
of $P$-coalgebras over $A$ (this is the dual statement of Theorem 1.2.5 in \cite{Fre2}).

Now we want to prove that for every $P$-coalgebra $A$ and every chain complex $C$, we have an isomorphism
of $P$-coalgebras $U_{P^*}(A)(C)\cong A\times P^*(C)$ where $\times$ is the product in $^P Ch_{\mathbb{K}}^+$. For this aim
we need the following lemma:
\begin{lem}
Let $A$ be a $P$-coalgebra and $C$ be a chain complex. The following equalizer defines the product
$A\times P^*(C)$ in the category of $P$-coalgebras:
\[
\xymatrix{ A\times P^*(C)=ker(d^0-d^1)\ar[r] & P^*(A\oplus C) \ar@<1ex>[r]^{d^0} \ar@<-1ex>[r]_{d^1} & P^*(P^*(A)\oplus C) \ar@/_1.3pc/[l]_{s^0} }
\]
where $d^0\mid_A=\rho_A$, $d^0\mid_C=id_C$, $d^1\mid_A=\Delta(A)$, $d_l\mid_C=id_C$, $s^0\mid_A=\epsilon(A)$, $s^0\mid_C=id_C$.
\end{lem}

\begin{proof}
We clearly have $s^0\circ d^0=s^0\circ d^1=id$ so $(d^0,d^1)$ is a reflexive pair in $^P Ch_{\mathbb{K}}^+$.
The space $ker(d^0-d^1)$ is the coreflexive equalizer of $(d^0,d^1)$ in $Ch_{\mathbb{K}}^+$ and is a
sub-$P$-coalgebra of $P^*(A\oplus C)$, so it is the coreflexive equalizer of $(d^0,d^1)$ in $^P Ch_{\mathbb{K}}^+$.
Let $X$ be a $P$-coalgebra, $u:X\rightarrow A$ a morphism of $P$-coalgebras and $v:X\rightarrow C$ a linear map.
They induce a map $(u,v):X\rightarrow A\oplus C$, hence a morphism of $P$-coalgebras $\varphi_{(u,v)}:X\rightarrow P^*(A\oplus C)$
obtained by the universal property of the cofree $P$-coalgebra. The proof ends by seeing that $\varphi_{(u,v)}$
admits a unique factorization through $ker(d^0-d^1)$.
\end{proof}

The coreflexive equalizer in $\Sigma$-modules defining the enveloping cooperad induces a coreflexive equalizer
in $P$-coalgebras
\[
\xymatrix{ U_{P^*}(A)(C) \ar[r] & P^*[A](C) \ar@<1ex>[r]^{d^0} \ar@<-1ex>[r]_{d^1} & P^*[P^*(A)](C) \ar@/_1.3pc/[l]_{s^0} }
\]
where $P^*[A](C)\cong P^*(A\oplus C)$, $P^*[P^*(A)](C)\cong P^*(P^*(A)\oplus C)$ and $d^0,d^1,s^0$ turn out to be
the morphisms of the lemma above. By unicity of the limit, we have proved the following result:
\begin{prop}
Let $A$ be a $P$-coalgebra and $C$ be a chain complex, then there is an isomorphism of $P$-coalgebras
\[
U_{P^*}(A)(C)\cong A\times P^*(C).
\]
\end{prop}

We also need the following general result about $\Sigma$-modules:

\begin{prop}
Let $M$ be a $\Sigma$-module and $C$ a chain complex. If $H_*(C)=0$ then $H_*(M(C))=H_*(M(0))$.
\end{prop}

\begin{proof}
Recall that we work over a field $\mathbb{K}$ of characteristic $0$.
We use the norm map $N:M(n)\otimes_{\Sigma_n} C^{\otimes n} \rightarrow M(n)\otimes C^{\otimes n}$ defined by
\[
N(c\otimes v_1\otimes...\otimes v_n)=\frac{1}{n!} \sum_{\sigma\in\Sigma_n} \sigma.c\otimes v_{\sigma(1)}\otimes...\otimes v_{\sigma(n)}.
\]
If we denote $p:M(n)\otimes C^{\otimes n}\rightarrow M(n)\otimes_{\Sigma_n} C^{\otimes n}$ the projection, then
\begin{eqnarray*}
(p\circ N)(c\otimes v_1\otimes...\otimes v_n) & = & \frac{1}{n!} \sum_{\sigma\in\Sigma_n} p(\sigma.c\otimes v_{\sigma(1)}\otimes...\otimes v_{\sigma(n)}) \\
 & = & \frac{1}{n!}\mid\Sigma_n\mid c\otimes v_1\otimes...\otimes v_n \\
 & = & c\otimes v_1\otimes...\otimes v_n
\end{eqnarray*}
so $p\circ N=id$. Therefore $M(n)\otimes_{\Sigma_n} C^{\otimes n}$ is a retract of $M(n)\otimes C^{\otimes n}$.
For $n\geq 1$, the Künneth formula gives us for every $k\geq 0$
\[
H_k(M(n)\otimes C^{\otimes n})=\bigoplus_{p+q=k} H_p(M(n)\otimes C)\otimes H_q(C^{\otimes n-1}).
\]
This is equal to $0$ for $n>1$ because the fact that $H_*(C)=0$ implies recursively that $H_*(C^{\otimes n})=0$
by the Künneth formula. This is also equal to $0$ for $n=1$ because the fact that $H_k(C)=0$ implies that
$H_k(M(1)\otimes C)=0$.
For $n=0$, we have $H_k(M(0))$. We conclude that $H_k(M(C))=H_k(M(0))$.
\end{proof}

We finally reach the crucial result of this section:
\begin{cor}
Let $A$ be a $P$-coalgebra and $C$ be a chain complex. If $H_*(C)=0$ then the canonical projection
$A\times P^*(C)\rightarrow A$ is a weak equivalence in $^P Ch_{\mathbb{K}}^+$.
\end{cor}

\begin{proof}
According to Proposition 2.8, we have $U_{P^*}(A)(C)\cong A\times P^*(C)$. We can apply Proposition 2.9
to the $\Sigma$-module $U_{P^*}(A)$ since $H_*(C)=0$ by hypothesis, so
\[
H_*(A\times P^*(C))=H_*(U_{P^*}(A)(C))=H_*(U_{P^*}(A)(0)).
\]
It remains to prove that $H_*(U_{P^*}(A)(0))=H_*(A)$. For this aim we show that $U_{P^*}(A)(0)\cong A$.
It comes from a categorical result: in any category with a final object and admitting products,
the product of any object $A$ with the final object is isomorphic to $A$. We apply this fact to
$U_{P^*}(A)(0)\cong A\times P^*(0)$. Indeed, the chain complex $0$ is final in $Ch_{\mathbb{K}}^+$
so $P^*(0)$ is final in $^P Ch_{\mathbb{K}}^+$.
\end{proof}

\subsection{Generating (acyclic) cofibrations, proofs of M4 and M5}

Before specifying the families of generating cofibrations and generating acyclic cofibrations,
we prove axioms M4 (i) and M5 (i). The cofibrantly generated structure will then be used
to prove M5 (ii) by means of a small object argument, slightly different from the preceding one
since we will use smallness with respect to injections systems.
The plan and some arguments parallel those of \cite{GG}. However, they work in cocommutative
differential graded coalgebras. Some care is necessary in our more general setting. This is the
reason why we give full details in proofs, in order to see where we can readily follow
\cite{GG} and where our modifications (for instance the notion of enveloping cooperad) step in.

\textbf{M5 (i).} We first need a preliminary lemma:
\begin{lem}
Every chain complex  $X$ of $Ch_{\mathbb{K}}^+$ can be embedded in a chain complex $V$ satisfying $H_*(V)=0$.
\end{lem}
This lemma helps us to prove the following result:
\begin{prop}
(i) Let $C$ be a $P$-coalgebra and $V$ be a chain complex such that $H_*(V)=0$. Then
the projection $C\times P^*(V)\rightarrow C$ is an acyclic fibration with the right lifting property
with respect to all cofibrations.

(ii) Every $P$-coalgebras morphism $f:D\rightarrow C$ admits a factorization
\[
D\stackrel{j}{\rightarrow} X \stackrel{q}{\rightarrow} C
\]
where $j$ is a cofibration and $q$ an acyclic fibration with the right lifting property
with respect to all cofibrations (in particular we obtain axiom M5 (i)).
\end{prop}

\begin{proof}
(i) According to Corollary 2.10, the map $C\times P^*(V)\rightarrow C$ is a weak equivalence
so it remains to prove that it has the right lifting property with respect to all cofibrations
(which implies in particular that it is a fibration). Let us consider the following commutative
square in $^P Ch_{\mathbb{K}}$:
\[
\xymatrix{ A \ar[d]_i \ar[r]^-a & C\times P^*(V) \ar[d] \\ B\ar[r]_b & C}
\]
where $i$ is a cofibration. A lifting in this square is equivalent to a lifting in each of the
two squares
\[
\xymatrix{ A \ar[d]_i \ar[r] & C \ar@{=}[d] \\ B \ar[r] & C}
\]
and
\[
\xymatrix{ A \ar[d]_i \ar[r] & P^*(V) \ar[d] \\ B \ar[r] & 0}.
\]

In the first square this is obvious, just take the bottom map $B\rightarrow C$ as a lifting.
In the second square, via the adjunction $U: {}^P Ch_{\mathbb{K}}^+\rightleftarrows Ch_{\mathbb{K}}^+:P^*$,
the lifting problem is equivalent to a lifting problem in the following square of $Ch_{\mathbb{K}}^+$:
\[
\xymatrix{ U(A) \ar[d]_{U(i)} \ar[r] & V \ar[d] \\ U(B) \ar[r] & 0}.
\]
The map $V\rightarrow 0$ is degreewise surjective so it is a fibration of $Ch_{\mathbb{K}}^+$,
which is acyclic because $H_*(V)=0$. The map $i$ is a cofibration, so $U(i)$ is a cofibration
by definition and has therefore the left lifting property with respect to acyclic fibrations.

(ii) According to Lemma 2.11, there exists an injection $i:U(D)\hookrightarrow V$ in $Ch_{\mathbb{K}}^+$
where $V$ is such that $H_*(V)=0$. Let us set $X=C\times P^*(V)$, $q:X\rightarrow C$ the projection and
\[
j=(f,\tilde{i}):D\rightarrow C\times P^*(V)
\]
where $\tilde{i}:D\rightarrow P^*(V)$ is the factorization of $i$ by universal property of the cofree $P$-coalgebra.
We have $q\circ j=f$. According to (i), the map $q$ is an acyclic fibration with the right lifting property with
respect to all cofibrations. It remains to prove that $j$ is a cofibration. Let us consider the composite
\[
D \stackrel{j}{\rightarrow} C\times P^*(V) \stackrel{pr_2}{\rightarrow} P^*(V) \stackrel{\pi}{\rightarrow} V
\]
where $pr_2$ is the projection on the second component and $\pi$ the projection associated to the cofree
$P$-coalgebra on $V$. We have $\pi\circ pr_2\circ j=\pi\circ\tilde{i}=i$ by definition of $\tilde{i}$.
The map $i$ is injective so $j$ is also injective, which implies that $U(j)$ is a cofibration in $Ch_{\mathbb{K}}^+$.
By definition it means that $j$ is a cofibration in $^P Ch_{\mathbb{K}}^+$.
\end{proof}

\textbf{M4 (i).} Let $p:X\rightarrow Y$ be an acyclic fibration, let us consider the commutative square
\[
\xymatrix{ C \ar[d]_i \ar[r]^a & X \ar[d]^p \\ D \ar[r]_b & Y}
\]
where $i$ is a cofibration. According to Proposition 2.12, the map $p$ admits a factorization $p=q\circ j$
where $j:X\rightarrow T$ is a cofibration and $q:T\rightarrow Y$ an acyclic fibration with the right lifting
property with respect to all cofibrations. Axiom M2 implies that $j$ is a weak equivalence.
Let us consider the commutative square
\[
\xymatrix{ X \ar[d]_j \ar@{=}[r] & X \ar[d]^p \\ T \ar[r]_q & Y}.
\]
According to axiom M4 (ii), there exists a lifting $r:T\rightarrow X$ in this square and $p$ is consequently
a retract of $q$ via the following retraction diagram:
\[
\xymatrix{ X \ar[d]_p \ar[r]^j & T \ar[d]_q \ar[r]^r & X \ar[d]^p \\ Y \ar@{=}[r] & Y \ar@{=}[r] & Y}.
\]
The fact that $p$ is a retract of $q$ implies that $p$ inherits from $q$
the property of right lifting property with respect to cofibrations. Indeed, we get the commutative diagram
\[
\xymatrix{C \ar[d]_i \ar[r]^a & X \ar[d]_p \ar[r]^j & T \ar[d]_q \ar[r]^r & X \ar[d]^p \\ D \ar[r]_b & Y \ar@{=}[r] & Y \ar@{=}[r] & Y}.
\]
In the square
\[
\xymatrix{ C \ar[d]_i \ar[r]^{j\circ a} & T \ar[d]^q \\ D \ar[r]_b & Y}
\]
there exists a lifting $h$ by property of $q$. Now let us define $\tilde{h}=r\circ h$.
Then
\[
\tilde{h}\circ i = r\circ h\circ i =r\circ j\circ a = a
\]
and
\[
p\circ \tilde{h} = p\circ r\circ h = q\circ h = b,
\]
so $\tilde{h}$ is the desired lifting.
This concludes the proof of M4 (i).

\textbf{Generating (acyclic) cofibrations} We first need two preliminary lemmas:
\begin{lem}
Let $C$ be a $P$-coalgebra. For every homogeneous element $x\in C$ there exists a sub-$P$-coalgebra $D\subset C$
of finite dimension such that $x\in D$ and $D_k=0$ for every $k>deg(x)$.
\end{lem}

\begin{proof}
Let $x\in C_n$ be an homogeneous element of degree $n$ and $p\in P(m)$.
The element $p$ gives rise to an operation $p^*:C\rightarrow C^{\otimes m}$. We have the formula
\begin{eqnarray*}
p^*(x) & = & \sum_{i_1+...+i_m=n}(\sum x_{i_1}\otimes...\otimes x_{i_m})\\
 & = & \sum_{(x)} x_{(1)}\otimes...\otimes x_{(m)}\\
\end{eqnarray*}
where the second line is written in Sweedler's notation.

We do a recursive reasoning on the degree $n$ of $x$. For $n=0$, we have $x=0$ since $C_0=0$
and $0$ is the trivial sub-$P$-coalgebra.
Now suppose the lemma true for every $k<n$. Let $x\in C_n$ and $p^*(x)$ as above, for a certain element $p\in P(m)$.
By hypothesis, there exist sub-$P$-coalgebras of finite dimension $D_{(1)}(p),...,D_{(m)}(p)$ such that for every
$1\leq j \leq m$, we have $x_{(j)}\in D_{(j)}(p)$ and $(D_{(j)}(p))_l=0$ for $l>deg(x_{(j)})$.
We set
\[
D(p)=\sum D_{(1)}(p)+...+\sum D_{(m)}(p).
\]
A finite sum of sub-$P$-coalgebras is stable under the operations of the $P$-coalgebra structure of $C$ and thus
form also a sub-$P$-coalgebra of $C$.
To build our sub-$P$-coalgebra $D$, we need to include $x$ and all its images under the operations induced by $P$.
We need also to include the image $dx$ of $x$ under the differential of $C$.
For this aim, we set
\[
D=\mathbb{K}.x\oplus D(dx)\oplus \sum_p D(p)
\]
where $D(dx)$ is the sub-$P$-coalgebra containing $dx$ given by our hypothesis (which exists since $deg(dx)<n$),
and the sum $\sum_p$ ranges over fixed bases of each $P(m)$ for every $m$.
This sum is actually finite. Indeed, each $P(m)$ is of finite dimension, and when $m>deg(x)$ we have $p^*(x)=0$ for degree reasons.
Consequently, the space $D$ is a sub-$P$-coalgebra of $C$ containing $x$.
It is of finite dimension as a finite sum of finite dimensional spaces.
Moreover, since for every $1\leq j\leq m$ we have $deg(x)>deg(x_{(j)})$, this construction
implies that $D_l=0$ if $l>deg(x)$.
\end{proof}

\begin{lem}
Let $j:C\rightarrow D$ be an acyclic cofibration and $x\in D$ a homogeneous element. Then there exists a sub-$P$-coalgebra $B\subseteq D$
such that:

(i) $x\in B$;

(ii) $B$ has a countable homogeneous basis;

(iii) the injection $C\cap B\hookrightarrow B$ is an acyclic cofibration in $^P Ch_{\mathbb{K}}^+$
(we denote also by $C$ the image of $C$ under $j$, since $j$ is injective and thus $j(C)\cong C$).
\end{lem}

\begin{proof}We want to define recursively sub-$P$-coalgebras
\[
B(1)\subseteq B(2) \subseteq...\subseteq D
\]
such that $x\in B(1)$, each $B(n)$ is finite dimensional and the induced map
\[
\frac{B(n-1)}{C\cap B(n-1)}\rightarrow \frac{B(n)}{C\cap B(n)}
\]
is zero in homology. This map is well defined, since we do the quotient by an intersection of two sub-$P$-coalgebras which is still
a sub-$P$-coalgebra. This property will be essential to prove the acyclicity of the injection the injection $C\cap B\hookrightarrow B$.

The $P$-coalgebra $B(1)$ is given by Lemma 2.13. Now suppose that for some integer $n\geq 1$ the coalgebra $B(n-1)$
has been well constructed. The space $B(n-1)$ is of finite dimension, so we can choose a finite set of homogeneous elements
$z_i\in B(n-1)$ giving cycles $\overline{z_i}\in \frac{B(n-1)}{C\cap B(n-1)}$, such that the homology classes of the $\overline{z_i}$
span $H_*(\frac{B(n-1)}{C\cap B(n-1)})$. The existence of an acyclic cofibration $j:C\rightarrow D$ implies that
$H_*(D/C)=0$. Consequently, there exists elements $z'_i\in D$ such that $\delta (z'_i)=z_i mod C$.
For every $i$, Lemma 2.13 provides us a finite dimensional sub-$P$-coalgebra
$A(z'_i)\subseteq D$ containing $z'_i$. We can then define
\[
B(n)=B(n-1)+\sum_i A(z'_i).
\]
The sub-$P$-coalgebra $B(n)$ is of finite dimension because it is the sum of finite dimensional sub-$p$-coalgebras.
We have $\delta(z'_i)-z_i\in C$ and  $z_i,z'_i\in B(n)$, so $\delta(z'_i)-z_i\in C\cap B(n)$,
hence the induced map in homology
\[
H_*(\frac{B(n-1)}{C\cap B(n-1)})\rightarrow H_*(\frac{B(n)}{C\cap B(n)})
\]
is zero because it sends the homology classes of the $\overline{z_i}$ to $0$.

Let us define $B=\bigcup B(n)$ and prove that $C\cap B\hookrightarrow B$ is an acyclic cofibration.
Firs it is injective so it is a cofibration. To prove its acyclicity, let us consider the following
short exact sequence:
\[
0\rightarrow B\cap C\hookrightarrow B\rightarrow \frac{B}{C\cap B}\rightarrow 0.
\]
It is sufficient to consider the long exact sequence induced by this sequence in homology and to
prove that $H_*(\frac{B}{C\cap B})=0$. Let $z\in B$ such that $\partial(\overline{z})=0$ in
$\frac{B}{C\cap B}$, where $\partial$ is the differential of $\frac{B}{C\cap B}$.
We have $\partial(z)\in B\cap C=\bigcup B(n)\cap C$ and $B(1)\subseteq...\subseteq D$
so there exists an integer $n$ such that $z\in B(n-1)$ and $\partial(z)\in B(n-1)\cap C$.
It implies that $[\overline{z}]\in H_*(\frac{B(n-1)}{C\cap B(n-1)})$, where $[\overline{z}]$ is
the homology class of $\overline{z}$. Thus $[\overline{z}]=0$ in $H_*(\frac{B(n)}{C\cap B(n)})$,
since the map $H_*(\frac{B(n-1)}{C\cap B(n-1)})\rightarrow H_*(\frac{B(n)}{C\cap B(n)})$ is zero
in homology. We deduce that $z=\partial(b)+B(n)\cap C$ for a certain $b\in B(n)$, so
$\overline{z}=\partial(\overline{b})$ in $\frac{B}{B\cap C}$ (the projection $x\mapsto \overline{x}$
commutes with the differentials). Finally, it means that every cycle of $\frac{B}{B\cap C}$ is
a boundary, i.e that $H_*(\frac{B}{C\cap B})=0$.
To conclude, the complex $B$ is a colimit over $\mathbb{N}$ of finite dimensional complexes
and thus has a homogeneous countable basis.
\end{proof}

Now we can give a characterization of generating cofibrations and generating acyclic cofibrations.

\begin{prop}
A morphism $p:X\rightarrow Y$ of $^P Ch_{\mathbb{K}}^+$ is

(i) a fibration if and only if it has the right lifting property with respect to the acyclic cofibrations
$A\hookrightarrow B$ where $B$ has a countable homogeneous basis;

(ii) an acyclic fibration if and only if it has the right lifting property with respect to the cofibrations
$A\hookrightarrow B$ where $B$ is of finite dimension.
\end{prop}

\begin{proof}
(i) One of the two implications is obvious. Indeed, if $p$ is a fibration then it has the right lifting property
with respect to acyclic cofibrations by definition. Conversely, suppose that $p$ has the right lifting property
with respect to the acyclic cofibrations $A\hookrightarrow B$ where $B$ has a countable homogeneous basis. We consider
the following lifting problem:
\[
\xymatrix{ C\ar[d]_j \ar[r]^f & X \ar[d]^p \\ D \ar@{-->}[ur] \ar[r] & Y}
\]
where $j$ is an acyclic cofibration. Let us define $\Omega$ as the set of pairs $(\overline{D},g)$,
where $\overline{D}$ fits in the composite of two acyclic cofibrations
\[
C\hookrightarrow \overline{D}\hookrightarrow D
\]
such that this composite is equal to $j$. The map $g:\overline{D}\rightarrow X$ is a lifting in
\[
\xymatrix{ C \ar[d] \ar[rr]^f & & X \ar[d]^p \\ \overline{D} \ar[r] & D \ar[r] & Y}.
\]
The collection $\Omega$ is really a set: it is the union of hom sets $Mor_{^P Ch_{\mathbb{K}}^+}(\overline{D},X)$
indexed by a subset of the set of subcomplexes $\overline{D}$ of $D$.
Recall that cofibrations are injective $P$-coalgebras morphisms. We endow $\Omega$
with a partial order defined by $(\overline{D_1},g_1)\leq (\overline{D_2},g_2)$
if $\overline{D_1}\subseteq\overline{D_2}$ and $g_2\mid_{\overline{D_1}}=g_1$.
The commutative square
\[
\xymatrix{ C\ar@{=}[d] \ar[rr]^f & & X \ar[d]^p \\ C \ar[r]_j & D\ar[r] & Y}
\]
admits $f$ as an obvious lifting, so $(C,f)\in\Omega$ and thus $\Omega$ is not empty.
Moreover, any totally ordered subset of $\Omega$ admits an upper bound, just take the
sum of its elements. We can therefore apply Zorn lemma. Let $(E,g)\in\Omega$ be a
maximal element. We know that $E$ is injected in $D$ by definition, and we want to
prove that $D$ is injected in $E$ in order to obtain $E=D$.

Let $x\in D$ be a homogeneous element. According to Lemma 2.14 applied to the acyclic
cofibration $E\hookrightarrow D$, there exists a sub-$P$-coalgebra
$B\subseteq D$ with a countable homogeneous basis such that $x\in B$ and $E\cap B\hookrightarrow B$ is an acyclic cofibration.
The lifting problem
\[
\xymatrix{ E\cap B \ar[d] \ar[r] & E \ar[r]^g & X \ar[d] \\ B\ar@{-->}[urr]^h \ar[r] & D\ar[r] & Y}
\]
admits a solution $h$ by hypothesis about $p$. We therefore extend $g$ into a map
$\tilde{g}:E+B\rightarrow X$ such that $\tilde{g}\mid_E=g$, $\tilde{g}\mid_B=h$.
According to the diagram above, we have $h\mid_{E\cap B}=g\mid_{E\cap B}$ so
$\tilde{g}$ is well defined.The short exact sequences
\[
0 \rightarrow E\cap B \rightarrow B \rightarrow \frac{B}{E\cap B} \rightarrow 0
\]
and
\[
0 \rightarrow E \rightarrow E+B \rightarrow \frac{E+B}{B} \rightarrow 0
\]
induce long exact sequences in homology
\[
...\rightarrow H_{n+1}(\frac{B}{E\cap B}) \rightarrow H_n(E\cap B) \rightarrow H_n(B) \rightarrow H_n(\frac{B}{E\cap B}) \rightarrow...
\]
and
\[
...\rightarrow H_{n+1}(\frac{E+B}{E}) \rightarrow H_n(E) \rightarrow H_n(E+B) \rightarrow H_n(\frac{E+B}{E}) \rightarrow...
\]
But $E\cap B\hookrightarrow B$ induces an isomorphism in homology so in the first exact sequence $H_*(\frac{B}{E\cap B})=0$.
Furthermore, the isomorphism $\frac{B}{E\cap B}\cong\frac{E+B}{E}$ implies that $H_*(\frac{E+B}{E})=0$.
Accordingly, the map $E\hookrightarrow E+B$ in the second exact sequence induces an isomorphism in homology, i.e
$E\hookrightarrow E+B$ is an acyclic cofibration. It means that $(E+B,\tilde{g})\in\Omega$, and by definition of
$\tilde{g}$ the inequality $(E,g)\leq (E+B,\tilde{g})$ holds in $\Omega$. Given that $(E,g)$ is supposed to be
maximal, we conclude that $E=E+B$, hence $x\in E$ and $E=D$. The map $g$ is the desired lifting,
and the map $p$ is a fibration.

(ii) If $p$ is an acyclic fibration, then $p$ has the right lifting property with respect to cofibrations according
to axiom M4 (i). Conversely, let us suppose that $p$ has the right lifting property with respect to cofibrations
$A\hookrightarrow B$ where $B$ is finite dimensional. The proof is similar to that of (i) with a slight change in
the definition of $\Omega$. Indeed, we consider the lifting problem
\[
\xymatrix{ C\ar[d]_j \ar[r]^f & X \ar[d]^p \\ D \ar[r] & Y}
\]
where $j$ is a cofibration. We define $\Omega$ as the set of pairs $(\overline{D},g)$ where $\overline{D}$
fits in a composite of cofibrations $C\hookrightarrow\overline{D}\hookrightarrow D$ such that this composite
is equal to $j$. We define the same partial order on $\Omega$ as in (i), and $\Omega$ is clearly not empty
since $(C,f)\in\Omega$. The set $\Omega$ is inductive so we can apply Zorn's lemma. Let $(E,g)$ be a maximal
element of $\Omega$, as before $E$ is injected in $D$ and we want to prove that $D$ is injected in $E$.
Let $x\in D$ be a homogeneous element, then there exists a finite dimensional
sub-$P$-coalgebra $B\subseteq D$ containing $x$ (this property is an adaptation to $P$-coalgebras of
Proposition 1.5 in \cite{GG}). The map $p$ has the right lifting property with respect to
$E\cap B\hookrightarrow B$ by hypothesis (since $E\cap B\hookrightarrow B$ is an injection of $P$-coalgebras,
thus a cofibration, and $B$ is of finite dimension), so the method of (i) works here. We extend $g$ to $\tilde{g}:E+B\rightarrow X$,
we have $(E+B,\tilde{g})\in\Omega$ and $(E,g)\leq (E+B,\tilde{g})$. The maximality of $(E,g)$ implies that
$E=E+B$ and $g:E=D\rightarrow X$ is the desired lifting.
\end{proof}

\textbf{M5 (ii).}We need here to use a refined version of the usual small object argument.
We use this expression to emphasize the fact that we will not use the usual simplifying hypothesis of smallness with respect
to any sequences of morphisms.
We will consider smallness only with respect to injections systems. Moreover, we will have to consider
colimits running over a certain ordinal bigger than $\mathbb{N}$.

We use the notations of section 1.3. In the case where $\mathcal{C}= {}^P Ch_{\mathbb{K}}^+$ and $\mathcal{D}$ is the collection of
injections of $P$-coalgebras, the colimit $colim_{\beta<\lambda} B(\beta)$ is a union $\bigcup_{\beta<\lambda} B(\beta)$.
We say that a $P$-coalgebra $A$ is small with respect to direct systems of injections if the map
\[
\bigcup_{\beta<\lambda}Mor_{^P Ch_{\mathbb{K}}^+}(A,B(\beta))\rightarrow Mor_{^P Ch_{\mathbb{K}}^+}(A,\bigcup_{\beta<\lambda} B(\beta))
\]
is a bijection.
Consider a morphism $f$ of $\mathcal{C}$ and a family of morphisms $\mathcal{F}=\{f_i:A_i\rightarrow B_i\}_{i\in I}$
such that the $A_i$ are small with respect to injections systems. If we can prove that the $i_{\beta}$
obtained in the construction of the $G^{\beta}(\mathcal{F},f)$ (see section 1.3 for the notation)
are injections, then we can use this refined
version of the small object argument to obtain a factorization $f=f_{\infty}\circ i_{\infty}$
where $f_{\infty}$ has the right lifting property with respect to the morphisms of $\mathcal{F}$
and $i_{\infty}$ is an injection (the injectivity passes to the transfinite composite).
This is the argument we are going to use to prove axiom M5 (ii) in $^P Ch_{\mathbb{K}}^+$.

Recall that the generating acyclic cofibrations of $^P Ch_{\mathbb{K}}^+$ are the
acyclic injections $j_i:A_i\hookrightarrow B_i$ of $P$-coalgebras such that the $B_i$
have countable homogeneous bases. In order to apply the refined small object argument explained above,
we need the following lemma:
\begin{lem}
Let $C$ be a object of $^P Ch_{\mathbb{K}}^+$ and $\kappa$ a cardinal. If $U(C)$ is $\kappa$-small with respect to injections systems, then
so is $C$ in $^P Ch_{\mathbb{K}}^+$.
\end{lem}

\begin{proof}
Let us consider a $\kappa$-filtered ordinal $\lambda$ and a $\lambda$-sequence of injections
\[
...\hookrightarrow B(\beta)\hookrightarrow B(\beta+1)\hookrightarrow ...
\]
of $P$-coalgebras, and let $f:C\rightarrow \bigcup_{\beta<\lambda} B(\beta)$ be a morphism of $P$-coalgebras.
The chain complex $U(C)$ is $\kappa$-small with respect to injections systems, so there exists
an ordinal $\beta<\lambda$ such that we have a unique factorization in $Ch_{\mathbb{K}}^+$
\[
f:C \stackrel{\tilde{f}}{\rightarrow} B(\beta) \hookrightarrow \bigcup_{\beta<\lambda} B(\beta).
\]
The map $f$ is a morphism of $P$-coalgebras and so is $B(\beta) \hookrightarrow \bigcup_{\beta<\lambda} B(\beta)$.
We deduce that $\tilde{f}$ is a morphism of $P$-coalgebras by the following argument.
For any cooperations $p^*_C:C\rightarrow C^{\otimes n}$, $p^*_{\beta}:B(\beta)\rightarrow B(\beta)^{\otimes n}$
and $p^*_{\lambda}:\bigcup_{\beta<\lambda} B(\beta)\rightarrow (\bigcup_{\beta<\lambda} B(\beta))^{\otimes n}$
associated to an element $p\in P^*(n)$, we consider the diagram
\[
\xymatrix{
C \ar[r]^-{\tilde{f}} \ar[d]_{p^*_C} & B(\beta) \ar[d]_{p^*_{\beta}} \ar@{^{(}->}[r]^-{i_{\beta}} & \bigcup_{\beta<\lambda} B(\beta) \ar[d]_{p^*_{\lambda}} \\
C^{\otimes n} \ar[r]_-{\tilde{f}^{\otimes n}} & B(\beta)^{\otimes n} \ar@{^{(}->}[r]_-{i_{\beta}^{\otimes n}} & (\bigcup_{\beta<\lambda} B(\beta))^{\otimes n}
}.
\]
The external rectangle commutes because $f$ is a morphism of $P$-coalgebras.
The right square commutes because $i_{\beta}$ is a morphism of $P$-coalgebras as a transfinite composite of morphisms of $P$-coalgebras.
We deduce that
\[
i_{\beta}^{\otimes n}\circ \tilde{f}^{\otimes n}\circ p^*_C = i_{\beta}^{\otimes n}\circ p^*_{\beta}\circ\tilde{f}.
\]
By injectivity of $i_{\beta}^{\otimes n}$ we get the commutativity of the left square
\[
\tilde{f}^{\otimes n}\circ p^*_C = p^*_{\beta}\circ\tilde{f},
\]
so $\tilde{f}$ is a morphism of $P$-coalgebras.
We have the desired factorization in $^P Ch_{\mathbb{K}}^+$.
\end{proof}

Let us consider the family of generating acyclic cofibrations $\mathcal{F}=\{j_i:A_i\hookrightarrow B_i\}_{i\in I}$.
According to Lemma 2.3.2 of \cite{Hov}, the $U(A_i)$ are $\kappa$-small for a certain cardinal $\kappa$. In particular,
they are $\kappa$-small with respect to injections systems. Lemma 2.16 implies that the $A_i$
are $\kappa$-small with respect to injection systems. Now, let $f:X\rightarrow Y$ be a morphism of $P$-coalgebras
and $\lambda$ a $\kappa$-filtered ordinal.
Recall that the construction of $G^{\beta}(\mathcal{F},f)$, $\beta<\lambda$, is given by a pushout
\[
\xymatrix{ \bigvee_i A_i \ar[d]_{\bigvee_i j_i} \ar[r] & G^{\beta-1}(\mathcal{F},f) \ar[d]^{i_{\beta}} \\
\bigvee_i B_i \ar[r] & G^{\beta}(\mathcal{F},f) }.
\]
The forgetful functor creates the small colimits, so we obtain the same pushout in
$Ch_{\mathbb{K}}^+$ by forgetting $P$-coalgebras structures. By definition of
cofibrations and weak equivalences in $Ch_{\mathbb{K}}^+$, given that $\bigvee_i j_i$
is an acyclic cofibration, the map $U(\bigvee_i j_i)$ is an acyclic cofibration in
$Ch_{\mathbb{K}}^+$. In any model category, acyclic cofibrations are stable by
pushouts, so the $U(i_{\beta})$ are acyclic cofibrations. By definition, it means that
the $i_{\beta}$ are acyclic cofibrations, i.e in our case acyclic injections of $P$-coalgebras.
We use our refined version of the small object argument to obtain a factorization $f=f_{\infty}\circ i_{\infty}$.
Injectivity and acyclicity are two properties which passes to the transfinite composite
$i_{\infty}$, so $i_{\infty}$ is an acyclic cofibration of $^P Ch_{\mathbb{K}}^+$.
Moreover, the map $f_{\infty}$ has by construction the right lifting property
with respect to the generating acyclic cofibrations and forms consequently a fibration.
Our proof is now complete.

\begin{rem}
This method provides us another way to prove M5 (i), by using this time the family of generating cofibrations.
\end{rem}

\section{The model category of bialgebras over a pair of operads in distribution}

Let $P$ be an operad in $Vect_{\mathbb{K}}$. Let $Q$ be an operad in $Vect_{\mathbb{K}}$ such that $Q(0)=0$, $Q(1)=\mathbb{K}$
and the $Q(n)$ are of finite dimension for every $n\in\mathbb{K}$. We suppose that there exists a mixed distributive law
between $P$ and $Q$ (see Definition 1.6). In the following, the operad $P$ will encode the operations of our bialgebras
and the operad $Q$ will encode the cooperations.

Recall that there exists a cofibrantly generated model category structure on the category $_P Ch_{\mathbb{K}}$
of $P$-algebras:
\begin{thm}(see \cite{Hin} or \cite{Fre3})
The category of $P$-algebras $_PCh_{\mathbb{K}}$ inherits a cofibrantly generated model category structure such that
a morphism $f$ of $_PCh_{\mathbb{K}}$ is

(i)a weak equivalence if $U(f)$ is a weak equivalence in $Ch_{\mathbb{K}}$, where $U$ is the forgetful functor;

(ii)a fibration if $U(f)$ is a fibration in $Ch_{\mathbb{K}}$;

(iii)a cofibration if it has the left lifting property with respect to acyclic fibrations.

\end{thm}
We can also say that cofibrations are relative cell complexes with respect to the generating cofibrations,
where the generating cofibrations and generating acyclic cofibrations are, as expected,
the images of the generating (acyclic) cofibrations
of $Ch_{\mathbb{K}}$ under the free $P$-algebra functor $P$.

Actually, this structure exists via a transfer of cofibrantly generated model category structure via the
adjunction $P:Ch_{\mathbb{K}}\rightleftarrows {}_P Ch_{\mathbb{K}}:U$ (see \cite{Hin} and \cite{Fre3}).
The forgetful functor creates fibrations and weak equivalences.
The free $P$-algebra functor $P$ preserves generating (acyclic) cofibrations by definition of the generating (acyclic) cofibrations of $_P Ch_{\mathbb{K}}$. Moreover, it preserves
colimits as a left adjoint (it is a general property of adjunctions, see \cite{MLa2} for instance). Thus it preserves
all (acyclic) cofibrations, which are relative cell complexes with respect to the generating (acyclic) cofibrations.
Such a pair of functors is called a Quillen adjunction, and induces an adjunction at the level of the associated homotopy
categories.
According to Theorem 1.9, we can lift this free $P$-algebra functor to the category of $Q$-coalgebras,
so the adjunction
\[
P:Ch_{\mathbb{K}}^+\rightleftarrows {}_P Ch_{\mathbb{K}}^+:U
\]
becomes an adjunction
\[
P: {}^Q Ch_{\mathbb{K}}^+\rightleftarrows {}^Q_P Ch_{\mathbb{K}}^+:U.
\]
Similarly, the adjunction
\[
U: {}^Q Ch_{\mathbb{K}}^+\rightleftarrows Ch_{\mathbb{K}}^+:Q^*
\]
becomes an adjunction
\[
U: {}^Q_P Ch_{\mathbb{K}}^+\rightleftarrows {}_P Ch_{\mathbb{K}}^+:Q^*.
\]
The model category structure on $(P,Q)$-bialgebras is then given by the following theorem:
\begin{thm}
The category $^Q_P Ch_{\mathbb{K}}^+$ inherits a cofibrantly generated model category structure
such that a morphism $f$ of $^Q_P Ch_{\mathbb{K}}^+$ is

(i) a weak equivalence if $U(f)$ is a weak equivalence in $^Q Ch_{\mathbb{K}}^+$
(i.e a weak equivalence in $Ch_{\mathbb{K}}^+$ by definition of the model structure on $^Q Ch_{\mathbb{K}}^+$);

(ii) a fibration if $U(f)$ is a fibration in $^Q Ch_{\mathbb{K}}^+$;

(iii) a cofibration if $f$ has the left lifting property with respect to acyclic fibrations.
\end{thm}
It is clear that this three classes of morphisms are stable by composition and contain the identity morphisms.
Axioms M2 and M3 are clear, axiom M4 (i) is obvious by definition of the cofibrations.
It remains to prove axioms M1, M4 (ii) and M5.

\textbf{M1.}The forgetful functor $U: {}^Q_P Ch_{\mathbb{K}}^+\rightarrow {}^Q Ch_{\mathbb{K}}^+$ creates
small limits. The proof is the same than in the case of $P$-algebras.
The forgetful functor $U: {}^Q_P Ch_{\mathbb{K}}^+\rightarrow {}_P Ch_{\mathbb{K}}^+$ creates the small
colimits. The proof is the same as in the case of $P$-coalgebras, see section 2.2.

\textbf{Generating (acyclic) cofibrations.}The treatment is similar to the case of $P$-algebras.
Let us note $\{j:A\hookrightarrow B\}$ the family of generating cofibrations
and $\{i:A\hookrightarrow B\}$ the family of generating acyclic cofibrations.
Then the $P(j)$ form the generating cofibrations of $^Q_P Ch_{\mathbb{K}}^+$ and the
$P(i)$ form the generating acyclic cofibrations:
\begin{prop}
Let $f$ be a morphism of $^Q_P Ch_{\mathbb{K}}^+$. Then

(i) $f$ is a fibration if and only if it has the right lifting property with respect to the $P(i)$,
where $i:A\hookrightarrow B$ an acyclic injection of $Q$-coalgebras such that $B$ has a countable homogeneous basis;

(ii) $f$ is an acyclic fibration if and only if it has the right lifting property with respect
to the $P(j)$, where $j:A\hookrightarrow B$ is an injection of $Q$-coalgebras such that $B$
is finite dimensional.
\end{prop}

\begin{proof}
This is a standard argument using only properties of the unit and the counit of the adjunction as well as lifting properties.
This is the same for instance than the one giving the generating (acyclic) cofibrations of the category of $P$-algebras.
It is part of more general results giving conditions to transfer a cofibrantly generated model structure via an adjunction,
see for instance theorem 11.3.2 in \cite{Hir}.
\end{proof}

\textbf{M4 (ii).}If M5 (ii) is proved, then M4 (ii) follows from M4 (i) and M5 (ii) by the same arguments as in the proof
of M4 (i) for coalgebras over an operad (see section 2).
Indeed, any acyclic cofibration $f$ admits a factorization $f=p\circ i$, where $p$ is a fibration
and $i$ an acyclic cofibration. The map $i$ has the left lifting property with respect to fibrations because it is
constructed via the small object argument (see the proof of M5), and $f$ is a retract of $i$ so it inherits this lifting property.

\textbf{M5.} The main difficulty here is to prove axiom M5. Let $f$ be a morphism of
$^Q_P Ch_{\mathbb{K}}^+$. Let us note $\mathcal{F}=\{P(j_i),j_i:A_i\hookrightarrow B_i\}_{i\in I}$
the family of generating cofibrations. Recall that the $A_i$ are $\kappa$-small with respect
to injections systems for a certain cardinal $\kappa$.
Actually, since the generating cofibrations have finite dimensional targets,
they are small with respect to sequences indexed by $\mathbb{N}$.
However, we write the proof without specifying $\kappa$, because after we use the same method to prove M5
for generating acyclic cofibrations, whose domains are small for a non countable cardinal.

We use the following Lemma to prove that the $P(A_i)$ are also $\kappa$-small with respect
to injections systems:
\begin{lem}
Let $C$ be an object of $^Q Ch_{\mathbb{K}}^+$. If $C$ is $\kappa$-small in $^Q Ch_{\mathbb{K}}^+$ with respect to injections systems,
then $P(C)$ is $\kappa$-small in $^Q _P Ch_{\mathbb{K}}^+$ with respect to injections systems.
\end{lem}

\begin{proof}
Let us suppose that $C$ is $\kappa$-small with respect to injections systems in $^Q Ch_{\mathbb{K}}^+$.
Let $\lambda$ be a $\kappa$-filtered ordinal and let $F:\lambda\rightarrow {}^Q_P Ch_{\mathbb{K}}^+$
be a $\lambda$-sequence whose arrows are injections.. For every $\beta<\lambda$,
\[
Hom_{^Q_P Ch_{\mathbb{K}}^+}(P(C),F(\beta))\cong Hom_{^Q Ch_{\mathbb{K}}^+}(C,(U\circ F)(\beta))
\]
hence
\begin{eqnarray*}
colim_{\beta<\lambda}Hom_{^Q_P Ch_{\mathbb{K}}^+}(P(C),F(\beta))& \cong & colim_{\beta<\lambda}Hom_{^Q Ch_{\mathbb{K}}^+}(C,(U\circ F)(\beta))\\
 & \cong & Hom_{^Q Ch_{\mathbb{K}}^+}(C,colim_{\beta<\lambda}(U\circ F)(\beta))
\end{eqnarray*}
because $U\circ F:\lambda\rightarrow {}^Q Ch_{\mathbb{K}}^+$ and $C$ is $\kappa$-small with respect to injections systems.
We can equip $colim_{\beta<\lambda}(U\circ F)(\beta)$ with a structure of $P$-algebra, such that with this structure
it forms $colim_{\beta<\lambda}F(\beta)$ in $^Q_P Ch_{\mathbb{K}}^+$. Indeed, we have
\[
colim_{\beta<\lambda}(U\circ F)(\beta)=\{[a],a\in F(\beta)\}/\sim
\]
where $a\sim b$ (i.e $[a]=[b]$), $a\in F(\beta),b\in F(\beta')$, $\beta\leq \beta'$, if the application $F(\beta)\rightarrow F(\beta')$ in
the $\lambda$-sequence sends $a$ to $b$. Let $[a_1],...,[a_r]\in colim_{\beta<\lambda}(U\circ F)(\beta)$ such that
$a_1\in F(\beta_1),...,a_r\in F(\beta_r)$. We consider $F(\beta)$ for a given ordinal $\beta\geq max(\beta_1,...,\beta_r)$ and we set,
for $\mu\in P(\beta)$, $\mu([a_1],...,[a_r])=\mu(a'_1,...,a'_r)$ where $a'_1,...,a'_r$ are representing
elements of $[a_1],...,[a_r]$ in $F(\beta)$. We then obtain a $P$-algebra structure on $colim_{\beta<\lambda}(U\circ F)(\beta)$
(one says that the forgetful functor creates the sequential colimits). Moreover, this $P$-algebra structure is compatible with
the $Q$-coalgebras structures of the $F(\beta)$, since it is defined via their $P$-algebra structures. It is therefore compatible
with the $Q$-coalgebra structure of $colim_{\beta<\lambda}(U\circ F)(\beta)$, such that we obtain the colimit in $^Q_P Ch_{\mathbb{K}}^+$.
We can finally write
\begin{eqnarray*}
colim_{\beta<\lambda}Hom_{^Q_P Ch_{\mathbb{K}}^+}(P(C),F(\beta)) & \cong & Hom_{^Q Ch_{\mathbb{K}}^+}(C,U(colim_{\beta<\lambda}F(\beta))) \\
 & \cong & Hom_{^Q_P Ch_{\mathbb{K}}^+}(P(C),colim_{\beta<\lambda}F(\lambda)).
\end{eqnarray*}
\end{proof}
We want to apply the small object argument to obtain a factorization $f=f_{\infty}\circ i_{\infty}$ of $f$.
Recall that, given a $\kappa$-filtered ordinal $\lambda$, for every $\beta<\lambda$, the space $G^{\beta}(\mathcal{F},f)$ is obtained by a pushout
\[
\xymatrix{ \bigvee_i P(A_i) \ar[d]_{\bigvee_i P(j_i)} \ar[r] & G^{\beta-1}(\mathcal{F},f) \ar[d]^{i_{\beta}} \\
\bigvee_i P(B_i) \ar[r] & G^{\beta}(\mathcal{F},f) }.
\]
The forgetful functor $U: {}^Q_P Ch_{\mathbb{K}}^+\rightarrow {}_P Ch_{\mathbb{K}}^+$ creates
small colimits, so we obtain the same pushout diagram in $_P Ch_{\mathbb{K}}^+$ by forgetting
the $Q$-coalgebras structures. The $j_i$ are cofibrations of $^Q Ch_{\mathbb{K}}^+$, so
the underlying chain complexes morphisms are cofibrations of $Ch_{\mathbb{K}}^+$.
Thus, via the adjunction $P:Ch_{\mathbb{K}}^+\rightleftarrows {}_P Ch_{\mathbb{K}}^+:U$,
the $P(j_i)$ are cofibrations of $_P Ch_{\mathbb{K}}^+$ and so are $\bigvee_i P(j_i)$.
In any model category, cofibrations are stable by pushouts, so the $i_{\beta}$ are cofibrations
of $_P Ch_{\mathbb{K}}^+$. By definition of cofibrations in $_P Ch_{\mathbb{K}}^+$, we can
apply Lemma 1.13 to $i_{\infty}$ to deduce that $i_{\infty}$ forms a cofibration
in $_P Ch_{\mathbb{K}}^+$. We now use the following proposition:
\begin{prop}
An (acyclic) cofibration of $_P Ch_{\mathbb{K}}^+$ forms an (acyclic) cofibration in $Ch_{\mathbb{K}}^+$.
\end{prop}

\begin{proof}
See section 4.6.3 in \cite{Hin} (Note that for a base field of characteristic zero,
every operad is $\Sigma$-split in the sense defined by Hinich).
\end{proof}
The maps $i_{\beta}$ (and thus $i_{\infty}$) forms therefore cofibrations in $Ch_{\mathbb{K}}^+$, i.e injections.
This is crucial to apply our version of the small object argument, since the $P(A_i)$ are $\kappa$-small
only with respect to injections systems. Finally, $i_{\infty}$ forms a cofibration in $^Q_P Ch_{\mathbb{K}}^+$.
The map $f_{\infty}$ has the right lifting
property with respect to the generating cofibrations and forms thus an acyclic fibration. Axiom M5 (i) is proved.

The method to prove M5 (ii) is the same up to two minor changes: we consider the family of generating
acyclic cofibrations, and use the stability of acyclic cofibrations under pushouts.

\bigskip{}

\end{document}